\documentclass[preprint]{elsarticle}
\usepackage{amssymb,amsmath,amsthm}
\usepackage{graphicx}
\usepackage{graphics}
\usepackage{enumerate}

\newcommand\cE{{\mathcal E}}

\newcommand\cH{{\mathcal H}}
\newcommand\cI{{\mathcal I}}

\newcommand\cM{{\mathcal M}}

\newcommand\cR{{\mathcal R}}

\newcommand\cU{{\mathcal U}}

\theoremstyle{plain}
\newtheorem{theorem}{Theorem}[section]
\newtheorem{lemma}[theorem]{Lemma}

\newtheorem{proposition}[theorem]{Proposition}

\theoremstyle{definition}

\newtheorem{claim}[theorem]{Claim}

\newcommand\lref[1]{Lemma~\ref{lem:#1}}
\newcommand\tref[1]{Theorem~\ref{thm:#1}}
\newcommand\cref[1]{Corollary~\ref{cor:#1}}
\newcommand\clref[1]{Claim~\ref{clm:#1}}

\newcommand\sref[1]{Section~\ref{sec:#1}}

\begin{document}

\title{The minimum number of vertices in uniform hypergraphs with given domination number}

\author{Csilla Bujt\'as\fnref{fn1}}
\ead{bujtas@dcs.uni-pannon.hu}
\address{Department of Computer Science and Systems
 Technology, University of Pannonia,
   \ 8200 Veszpr\'em, Egyetem u.\ 10, Hungary and Alfr\'ed R\'enyi Institute of Mathematics, Hungarian Academy of Sciences, P.O.B. 127, Budapest H-1364, Hungary.}
   
\author{Bal\'azs Patk\'os\fnref{fn2}}
\ead{patkos@renyi.hu}
\address{Alfr\'ed R\'enyi Institute of Mathematics, Hungarian Academy of Sciences, P.O.B. 127, Budapest H-1364, Hungary.}

\author{Zsolt Tuza\fnref{fn1}}
\ead{tuza@dcs.uni-pannon.hu}
\address{Department of Computer Science and Systems
 Technology, University of Pannonia,
   \ 8200 Veszpr\'em, Egyetem u.\ 10, Hungary and Alfr\'ed R\'enyi Institute of Mathematics, Hungarian Academy of Sciences, P.O.B. 127, Budapest H-1364, Hungary.}

\author{M\'at\'e Vizer\fnref{fn1}}
\ead{vizermate@gmail.com}
\address{Alfr\'ed R\'enyi Institute of Mathematics, Hungarian Academy of Sciences, P.O.B. 127, Budapest H-1364, Hungary.}

\fntext[fn1]{Research supported in part by the National Research, Development and Innovation
Office -- NKFIH under the grant SNN 116095.}

\fntext[fn2]{Research supported in part by the National Research, Development and Innovation
Office -- NKFIH under the grant SNN 116095 and by
 the J\'anos Bolyai Research Scholarship of the Hungarian Academy of Sciences}

\begin{abstract}
The \textit{domination number} $\gamma(\cH)$ of a hypergraph $\cH=(V(\cH),\cE(\cH))$ is the minimum size of a subset $D\subset V(\cH)$ of the vertices such that for every $v\in V(\cH)\setminus D$ there exist a vertex $d \in D$ and an edge $H\in \cE(\cH)$ with $v,d\in H$. We address the problem of finding the minimum number $n(k,\gamma)$ of vertices that a $k$-uniform hypergraph $\cH$ can have if $\gamma(\cH)\ge \gamma$ and $\cH$ does not contain isolated vertices. We prove that $$n(k,\gamma)=k+\Theta(k^{1-1/\gamma})$$ and also consider the $s$-wise dominating and the distance-$l$ dominating version of the problem. In particular, we show that the minimum number $n_{dc}(k,\gamma, l)$ of vertices that a connected $k$-uniform hypergraph with distance-$l$ domination number $\gamma$ can have is roughly $\frac{k\gamma l}{2}$.
\end{abstract} 

\maketitle

\section{Introduction}
In this paper we establish basic inequalities involving fundamental hypergraph parameters such as order, edge size, and domination number.

Many problems in extremal combinatorics are of the following form: 
what is the smallest or largest size that a graph, hypergraph, set system can have, provided it satisfies a prescribed property? In most cases, size is measured by the number of edges, hyperedges, sets, respectively, contained in the object, and the number of vertices is usually included in the prescribed property. However, sometimes it can be interesting and even applicable to consider problems about the minimum or maximum number of vertices \cite{NP,T1,T2}.

In the present paper we address the problem of finding the minimum number of vertices in a $k$-uniform hypergraph that has a large domination number. The domination number $\gamma(G)$ of a graph $G$, a widely studied notion (see \cite{HHS1}, \cite{HHS2}), is the smallest size that a subset $D \subset V(G)$ of the vertices can have if every vertex $v \in V(G)\setminus D$ has a neighbor in $D$. 

We will be interested in the hypergraph version of this notion, which was investigated first in \cite{A1} and later studied in \cite{A2,ABBT,BHT,HLo,JT}. 
Let $\cH=(V(\cH), \cE(\cH))$ be a hypergraph.
The \textit{neighborhood}\footnote{In this paper we use the short term  ``neighborhood'', although this is called ``closed neighborhood'' in the main part of the literature. We note that the inclusion of $\{v\}$ in the definition of $N_v$ may be omitted if $v$ is not an isolated vertex in $\cH$.}  of a vertex $v \in V(\cH)$
 is the set $N_v:=\{v\} \cup \bigcup _{E \in \cE(\cH)\,:v \in E} E$, and the
 \textit{neighborhood  of a set} $S\subset V(\cH)$
 is defined as $N(S):=\bigcup _{v \in S} N_v$. 
A set $D \subset V(\cH)$ is called a \textit{dominating set} of $\cH$ if $D \cap N_v \neq \emptyset$ for all $v \in V(\cH)$. Equivalently we can say that $D$ is a dominating set if and only if $N(D) = V(\cH)$. The minimum size $\gamma(\cH)$ of a dominating set in a hypergraph $\cH$ is \textit{the domination number} of $\cH$.
As all isolated vertices always are contained in every dominating set, they can be eliminated in an obvious way, therefore we restrict our attention to hypergraphs without isolates.

Let $n(k,\gamma)$ be the minimum number of vertices that a $k$-uniform hypergraph with no isolated vertices must contain if its domination number is at least $\gamma$. Beyond the trivial case of $n(k,1)=k$, the problem of determining $n(k,\gamma)$ is natural and seems to be interesting enough to be addressed on its own right; nevertheless, Gerbner et al.\ (Problem 17 in \cite{Getal}) arrived from a combinatorial search-theoretic framework at the particular problem of deciding whether $n(k,3)\ge 2k+3$ holds or not. We answer this problem in the negative, 
 determining the asymptotic behavior of $n(k,\gamma)$
 as a function of $k$ for every fixed $\gamma$,
 up to the exact growth order of the second term.
To state our result in full strength, we need to introduce two generalizations of domination. For an integer $s>0$ we call $D \subset V(\cH)$ an \textit{s-dominating set} of $\cH$ if $|D \cap N_v| \ge s$ for all $v \in V(\cH)\setminus D$ and we call $D$ an \textit{s-tuple dominating set} if $|D \cap N_v| \ge s$ for all $v \in V(\cH)$. Note that dominating sets are exactly the $1$-dominating sets and $1$-tuple dominating sets.
As introduced in \cite{FJ} and \cite{HH}, respectively, the minimum size $\gamma(\cH,s)$ of an $s$-dominating set in a hypergraph $\cH$ is the \textit{$s$-domination number of\/ $\cH$} and the minimum size $\gamma_{\times}(\cH,s)$ of an $s$-tuple dominating set in a hypergraph $\cH$ is the \textit{$s$-tuple domination number\/\footnote{The standard notation for $s$-tuple domination in the graph theory literature is $\gamma_{\times s}(G)$, but for the different variants of domination in this paper we try to use notations which are similar to each other in their form, this is why we put $s$ in another position.} of\/ $\cH$}. By definition, we have $\gamma(\cH,s)\le \gamma_{\times}(\cH,s)$. For every pair 
 $\gamma,s$
of integers with $\gamma \ge s$, let $n(k,\gamma,s)$ denote the minimum number of vertices that a $k$-uniform hypergraph $\cH$ must have if $\gamma(\cH,s)\ge \gamma$  holds and there exist no isolated vertices in $\cH$ and let $n_{\times}(k,\gamma,s)$ denote the minimum number of vertices that a $k$-uniform hypergraph $\cH$ must have if $\gamma_{\times}(\cH,s)\ge \gamma$  holds and there exist no isolated vertices. From the above, we have $n_{\times}(k,\gamma,s)\le n(k,\gamma,s)$.

\vspace{3mm}

Our main theorem about $s$-domination is the following.

\begin{theorem}
\label{thm:main}
For every $\gamma \ge 2$ and $ s\ge 1$ with $\gamma > s$ we have $$k+k^{1-1/(\gamma-s+1)} \le n_{\times}(k,\gamma,s)\le n(k,\gamma,s) \le k+(4+o(1))k^{1-1/(\gamma-s+1)}.$$ 
\end{theorem}
 
Another generalization of domination is distance-$l$ domination, which was introduced by Meir and Moon in \cite{MM}. This notion has been studied only for graphs so far. A good survey of the results until 1997 is \cite{H}. For more recent upper and lower bounds on the distance-$l$ domination number of graphs see \cite{HL} and \cite{DFHK}. 

In distance-$l$ domination a vertex $v$ dominates all  vertices that are at distance at most $l$ from $v$. As the definition of distance in graphs involves paths, and paths in hypergraphs can be defined in several ways, distance-$l$ domination could be addressed with each of those definitions. But as we will remark in \sref{discuss}, only so-called `Berge paths' offer new problems in our context. A \textit{Berge path} of length $l$ is a sequence $v_0,H_1,v_1,H_2,v_2,\dots,H_l,v_l$ with $v_i \in V(\cH)$ for $i=0,1,...,l$ and $v_{i-1},v_i\in H_i\in \cE(\cH)$ for $i=1,2,...,l$. The distance $d_{\cH}(u,v)$ of two vertices $u,v \in V(\cH)$ is the length of a shortest Berge path from $u$ to $v$.
 The \textit{ball centered at $u$ and of radius $l$} consists of those vertices of $\cH$ which are at distance at most $l$ from $u$; it will be denoted by $B_l(u)$. We call $D \subset V(\cH)$ a \textit{distance-$l$ dominating set} of $\cH$ if $\bigcup_{u \in D} B_l(u)=V(\cH)$. Equivalently we can say that $D \subset V(\cH)$ is a distance-$l$ dominating set if and only if $D \cap B_l(v) \neq \emptyset$ for all $v \in V(\cH)$. Note that distance-$1$ dominating sets are the usual dominating sets.

The minimum size of a distance-$l$ dominating set in a hypergraph $\cH$ is the distance-$l$ domination number $\gamma_d(\cH,l)$. Let further $n_d(k,\gamma,l)$ denote the minimum number of vertices that a $k$-uniform  hypergraph $\cH$ with no isolated vertices can contain if $\gamma_d(\cH,l)\ge \gamma$ holds. The next proposition shows that $n_d(k,\gamma,l)$ does not depend on $l$ once $l \ge 2$ is supposed.  

\begin{proposition}
\label{thm:disteasy} For any $k, l \ge 2$ and $\gamma \ge 1$ we have $n_d(k,\gamma,l)=k\gamma$, and the unique extremal hypergraph consists of $\gamma$ pairwise disjoint edges.
\end{proposition}    

\begin{proof}
It is clear that the $k$-uniform hypergraph with just $\gamma$ disjoint edges yields the upper bound $n_d(k,\gamma,l)\le k\gamma$.


We prove the lower bound by induction on $\gamma$. The case $\gamma=1$ is trivial. So assume that $\gamma \ge 2$, and let $\cH=(V(\cH),\cE(\cH))$ be a $k$-uniform hypergraph with $\gamma_d(\cH,l) \ge \gamma$. Consider an arbitrary $v \in V(\cH)$. Any vertex in $N(B_{l-1}(v))$ is distance-$l$ dominated by $v$, therefore the $k$-uniform hypergraph $\cH'$ induced by the edge set $\{H \in \cE(\cH): H\cap B_{l-1}(v)=\emptyset\}$ covers all vertices of $\cH$ not distance-$l$ dominated by $v$. The assumption $\gamma_d(\cH,l) \ge \gamma$ implies $\gamma_d(\cH',l) \ge \gamma-1$ and thus using that $|B_{l-1}(v)|\ge k$ for $l \ge 2$ and by induction we obtain $$|V(\cH)| =|B_{l-1}(v)|+|V(\cH')|\ge k+(\gamma-1)k=\gamma k.$$
Strict inequality holds whenever $v$ has degree at least two.
\end{proof}

The problem becomes more interesting when disconnected hypergraphs get excluded.
Hence, for $k \ge 2$ and $l,\gamma \ge 1$ let $n_{dc}(k,\gamma,l)$ denote the minimum number of vertices that a $k$-uniform $\mathbf{connected}$ hypergraph $\cH$ must contain if it has $\gamma_d(\cH,l)\ge \gamma$. 

\vspace{3mm}

To state our main result concerning $n_{dc}(k,\gamma,l)$ we need to define the following function:
\[ f(k,\gamma,l) :=
  \begin{cases}
    \frac{l}{2}k\gamma +\max\{k,\gamma\}       & \quad \text{if } l \text{ is even,}\\
    \frac{l+1}{2}k\gamma  & \quad \text{if } l \text{ is odd.}\\
  \end{cases}
\]

\vskip 0.4truecm

\begin{theorem}
\label{thm:dist_conn} 

\textbf{(a)} For any $k,l\ge 2$ we have $$\frac{(2l+1)k}{2} \le n_{dc}(k,2,l) \le \min\left\{\left\lceil\frac{(2l+1)(k+1)}{2} \right\rceil, (l+1)k \right\}.$$ 

\textbf{(b)} For any $k \ge 2$, $l\ge 4$ and $\gamma\ge 3$ we have $$k\left\lceil \left(\frac{l-1}{2}-1\right)\gamma\right\rceil<n_{dc}(k,\gamma,l) \le f(k,\gamma,l).$$

\textbf{(c)} For any $k \ge 2$ and $\gamma\ge 3$ we have 
$$k\gamma\le n_{dc}(k,\gamma,2)\le k\gamma+\max\{k,\gamma\}.$$

\textbf{(d)} For any $k \ge 2$ and $\gamma\ge 3$ we have 
$$k\gamma\le n_{dc}(k,\gamma,3)\le 2k\gamma.$$
\end{theorem}

\vspace{1cm}

The remainder of the paper is organized as follows: we prove \tref{main} in Section~2, and \tref{dist_conn} in Section 3. Section 4 contains some final remarks, also including a general upper bound on $\gamma_{dc}(\cH,l)$ as a function of $l$, the number of vertices, and the edge size.

\section{Proof of \tref{main}}

In this section we prove our bounds on $n_{\times}(k,\gamma,s)$ and $n(k,\gamma,s)$. 
First we verify the bound $k+k^{1-1/(\gamma-s+1)} \le n_{\times}(k,\gamma,s)$. Observe that it is enough to prove the statement for $s=1$, since for any hypergraph $\cH$ we have $\gamma_{\times}(\cH,s)-(s-1) \ge \gamma_{\times}(\cH,1)$ as for any $s$-tuple dominating set $D$ of $\cH$ and a $s'$-subset $D'$ of $D$ the set $D\setminus D'$ $(s-s')$-tuple dominates $\cH$. Consequently $$n_{\times}(k,\gamma,s) \ge n_{\times}(k,\gamma-(s-1),1),$$ which implies the statement.

To see $n(k,\gamma,1) \ge k+k^{1-1/\gamma}$ let $\cH$ be a $k$-uniform hypergraph with $\gamma(\cH)\ge \gamma \ge 2$. Let $G=(V(\cH),E)$ be the graph with $(u,v)\in E$ if and only if no $H\in \cE(\cH)$ contains both $u$ and $v$. The $\gamma \ge 2$ condition means that for any vertex $v\in V(\cH)$ there exists a $u$ such that no edge $H\in \cE(\cH)$ contains both $u$ and $v$, thus $G$ does not contain any isolated vertices. Let us write $n=|V(\cH)|=|V(G)|=k+x$ and let $t$ be the number of edges in a largest matching $M=(V(M), \cE(M))$ of $G$. Note that two distinct vertices $u',v'$ outside $V(M)$ cannot be adjacent to two distinct endpoints $u,v$ of an edge $e\in \cE(M)$ as the matching $(M\setminus \{e\})\cup\{(u,u'),(v,v')\}$ would contradict the maximality of $M$. Then either just one of $u$ and $v$ has neighbors outside $M$, or none of them have any, or they share their unique neighbor outside $M$. We denote by $e(v)$ the (or an)  endpoint of $e$ whose `outside' neighborhood in this sense contains the `outside' neighborhood of the other endpoint, and let $d_{e(v)}$ denote the size of $N_{e(v)} \setminus V(M)$.

By the definition of $\gamma =\gamma_{\times}(\cH,1)=\gamma (\cH, 1)$ and $G$ we have that for any set $\Gamma$ of $\gamma-1$ vertices in $V(G)$ there is a vertex $v \in V(G)$ which is connected by edges in $E(G)$ to all the vertices of $\Gamma$. If $\Gamma$ is a subset of $V(G) \setminus V(M)$, then the vertex which is adjacent to all vertices of $\Gamma$  must be in $V(M)$, since $M$ is maximal. By this we obtain$$\sum_{e \in \cE(M)} \binom{d_{e(v)}}{\gamma-1} \ge \binom{|V(G) \setminus V(M)|}{\gamma-1}.$$

Writing $d:=\max_{e \in \cE(M)}d_{e(v)}$ the above inequality yields $$td^{\gamma -1} \ge (k+x-2t)^{\gamma -1},$$
and rearranging gives $$ d \ge \frac{k+x-2t}{t^{ \frac{1}{\gamma -1}}}.$$
Let $e\in \cE(M)$ be an edge with $d_{e(v)}=d$, and let $H$ be any hyperedge $H \in \cE(\cH)$ containing $e(v)$. Just as any hyperedge, $H$ must avoid an endpoint of each edge in $M$, and $H$ is disjoint from $N_{e(v)} \setminus V(M)$.
Therefore, we obtain $k+x=n \ge d+t+k$ and thus $x\ge d+t$.
Plugging the previous inequality into this and rearranging yields: $$ t^{\frac{1}{\gamma-1}}(x-t+2t^{\frac{\gamma-2}{\gamma-1}}) \ge k+x.$$
Now using that $x \ge t$ and $t \ge t^{\frac{\gamma-2}{\gamma-1}}$, we obtain that the left-hand side of the previous inequality is at most $x^{\frac{\gamma}{\gamma-1}} +x$ and therefore we have $$x^{\frac{\gamma}{\gamma-1}} +x \ge k+x,$$
which proves the required lower bound.

To prove the bound $n(k,\gamma,s) \le k+(4+o(1))k^{1-1/(\gamma-s+1)}$ we need a construction. This involves projective geometries or linear vector spaces over finite fields. We will use the Gaussian or $q$-binomial coefficient ${n \brack k}_q$ that denotes the number of $k$-dimensional subspaces of a vector space of dimension $n$ over $\mathbb{F}_q$, i.e. $${n \brack k}_q:=\frac{\prod_{i=1}^k(q^{n-i+1}-1)}{\prod_{i=1}^k(q^i-1)}$$
and we will omit $q$ from the subscript when it is clear from the context.
Let $q$ be a prime power, $t$ be any positive integer and $U$ be a $\gamma$-dimensional vector space over $\mathbb{F}_q$. 
Let $E_1,E_2,\dots,E_m$ be the 1-dimensional subspaces of $U$ and $U_1,U_2,\dots,U_m$ the $(\gamma-1)$-dimensional subspaces of $U$, where $m={\gamma \brack 1}_q={\gamma \brack \gamma-1}_q=q^{\gamma-1}+q^{\gamma-2}+\dots+1$. 
Let $A_1,A_2,\dots,A_m,B$ be pairwise disjoint sets with $B=\{b_1,b_2,\dots,b_m\}$ and $|A_i|=t$ for all $1\le i \le m$. Let us define $\cH_{q,\gamma,t}=\{H_1,H_2,\dots, H_m\}$ by 
$$H_i:=\{b_i\}\cup \bigcup_{j: E_i \not\le U_j}A_j.$$
We claim that $\gamma(\cH_{q,\gamma-s+1,t},s)\ge \gamma$. 
Suppose not and let $D=D_B \cup D_A$ 
be a minimal $s$-dominating set of $\cH=\cH_{q,\gamma-s+1,t}$ with $D_B=D \cap B$, $D_A=D\setminus D_B$ and $|D|<\gamma$. As every vertex $d\in D_B$ is contained in exactly one hyperedge $H_d$ of $\cH$, each such $d$ can be replaced by a vertex $d'\in V(\cH)\setminus (D\cup B)$ to obtain an $s$-dominating set $D'$ with $D'\subseteq V(\cH)\setminus B$ and $|D'|=|D|<\gamma$. Let $D'=\{d_1,d_2,\dots,d_p\}$ and $D''=\{d_1,d_2,\dots,d_{\gamma-s}\}$. Then for $Z=\bigcap_{j:\exists v \in D''\cap A_j} U_j$ we obtain 
$$
\dim(Z)\ge 1.
$$
If $E$ is a $1$-subspace of $Z$, then the corresponding vertex $b \in V(\cH)$ is not dominated by any vertex $d\in D''$ and thus at most $(s-1)$-dominated by $D'$, which is a contradiction.

Let us consider the other parameters of the above hypergraph: $n=|V(\cH_{q,\gamma-s+1,t})|=m(t+1)$ and $\cH_{q,\gamma-s+1,t}$ is $k_{q,\gamma-s+1,t}$-uniform with $k_{q,\gamma-s+1,t}=1+q^{\gamma-s}t$, therefore if $t=q$, then we obtain $n=q^{\gamma-s+1}+2(q^{\gamma-s}+q^{\gamma-s-1}+\dots +q)+1$ and $k_\gamma=k_{q,\gamma-s+1,q}=1+q^{\gamma-s+1}$, thus we have $n\le k_{\gamma-s+1}+4k_{\gamma-s+1}^{1-1/(\gamma-s+1)}$. This finishes the proof of the upper bound if $k$ is one larger than the $(\gamma-s+1)$st power of a prime.

 Finally, let us consider the general case when $k'=1+q^{\gamma-s+1}+e$ with $e<q'^{\gamma-s+1}-q^{\gamma-s+1}$ where $q'$ is the smallest prime larger than $q$. It is well-known that $q'=q+o(q)$ and thus $e=o(q^{\gamma-s+1})$. Let $C_1,C_2,\dots, C_{q+1}$ be pairwise disjoint sets all of size $\lceil\frac{e}{q-\gamma+s+2}\rceil$, all being disjoint from $V(\cH_{q,\gamma-s+1,q})$. We renumber the subspaces $U_1,U_2,\dots, U_m$ in such a way that $U_1,U_2, \dots, U_{q+1}$ correspond to the dual of a $(q+1)$-arc in $PG(\gamma-s,q)$, i.e.\ every 1-subspace $E$ of $V$ is contained in at most $\gamma-s-1$ subspaces among $U_1,U_2,\dots,U_{q+1}$. (For a general introduction to finite geometries, see \cite{Hirsch}.) Therefore, for any $1\le i \le m$, the sets $I_i:=\{j:E_i \not\le U_j, 1\le j\le q+1\}$ satisfy $|I_i|\ge q-\gamma+s+2$ and thus there exists a set $T_i\subset \bigcup_{j\in I_i}C_j$ of size $e$. Let us define
 $$H'_i:=\{b_i\}\cup \bigcup_{j: E_i \not\le U_j}A_j \cup T_i.$$ By definition we have $|H'_i|=k'$ for all $i=1,2,\dots,m$. The $s$-domination number of the new hypergraph is the same as that of the old one, as for any $v \in C_i$ and $u\in A_i$ we have $N_u \subset N_v$. Moreover the number $n'$ of vertices in the new hypergraph is $$n+\lceil\frac{e}{q-\gamma+s+1}\rceil(q+1)\le k+4k^{1-1/(\gamma-s+1)}+e+O_\gamma(e/q)\le k'+(4+o(1))k'^{1-1/(\gamma-s+1)},$$
as $O_\gamma(e/q)=o(q^{\gamma-s})$ holds by $e=o(q^{\gamma-s+1})$.
\hfill $\qed$

\section{Distance domination}
\label{sec:dist}
In this section we prove \tref{dist_conn}, the lower and upper bounds on $n_{dc}(k,\gamma,l)$. 


\subsection{The $j$-radius of trees}
We start with some definitions and an auxiliary statement that we will use in the proof. 

\bigskip

\noindent \textbf{Definition.} 
For positive integers $a_1, a_2,...,a_{h}$ the \textit{spider graph}, denoted by $$S(a_1,a_2,\dots, a_{h}),$$
is the tree
on $1+\sum_{i=1}^ha_i$ vertices which is obtained 
from $h$ paths of lengths
$a_1,a_2,\dots, a_{h}$, respectively, by identifying the first vertices of those paths to a single vertex $v$ of degree $h$.
Hence, $S(a_1,a_2,\dots, a_{h})\setminus \{v\}$ has $h$ connected components, say $C_1,C_2,\dots,C_{h}$, where each $C_i$ is a path $P_{a_i}$ on $a_i$ vertices (for $i=1,2,...,h$).

\vskip 0.3truecm

 In a connected graph $G=(V(G),E(G))$, the \textit{excentricity} of a vertex $v\in V(G)$ is defined as
  $${\rm exc}_G(v):=\max\{d_G(u,v):u\in V(G)\}$$
  and let the \textit{radius} of $G$ be $$r(G):=\min\{{\rm exc}_G(v):v\in V(G)\}.$$ 
More generally, for any $\emptyset \neq W \subset V(G)$ let us write $${\rm exc}_G(W):=\max\{\min\{d_G(u,w): w \in W\} : u \in V(G)\}$$  and for an integer $j \ge 1$ let the \textit{$j$-radius of $G$} be $$r_j(G):=\min\{{\rm exc}_G(W):W\subset V(G), |W|\le j\}.$$ 
Certainly we have $r(G)=r_1(G)$. Finally, let $$r_j(n):=\max\{r_j(T): |V(T)|=n, \ T ~\textnormal{is a tree}\}.$$

\vspace{2mm}

\noindent The numerical bounds themselves in the next lemma concerning the radius of a tree are folklore; for later use, however, we need a more detailed assertion which describes some structural properties, too. Some bounds on the function $r_j(n)$ can be derived also from results of Meir and Moon \cite{MM}, but the following is a little sharper.

\begin{lemma}
\label{lem:genradius} Let $n\ge j$ be positive integers. Then we have
$$\left\lfloor\frac{n}{j+1}\right\rfloor\le r_{j}(n)\le \left\lceil \frac{n}{j+1}\right\rceil.$$ Moreover, $r_1(n)=\left\lceil \frac{n-1}{2}\right\rceil$ and

\vspace{2mm}

\textbf{(i)} if $n$ is even, then the only tree with $r_1(T)=\lceil \frac{n-1}{2}\rceil$ is the path $P_n$ on $n$ vertices.

\vspace{2mm}

\textbf{(ii)} If $n$ is odd and $r_1(T)=\lceil \frac{n-1}{2}\rceil$ holds, then $T$ is a path $P_{n-1}$ with a pendant edge. Furthermore, $T$ contains two copies of $P_{n-1}$ if and only if $T$ is either a path $P_n$ or a fork $F_n$. Otherwise $T$ contains just one copy of $P_{n-1}$.
\end{lemma}

\begin{proof}
Let us first prove the statements about $r_1(n)$. Let $T$ be an arbitrary tree on $n$ vertices and let $v$ be a middle vertex of a longest path $P$ in $T$. If $P$ contains $l$ vertices, then any vertex is at distance at most $\lceil \frac{l-1}{2}\rceil$ from $v$. This implies all assertions of the lemma if $n$ is even. If $n$ is odd, this implies that $T$ must contain a path on $n-1$ vertices and thus $T$ is a path $P_{n-1}$ and a pendant edge.

Let us now prove the general lower bound. We claim that $$\left\lfloor \frac{n}{j+1}\right\rfloor = r_{j}\left(S\left(\left\lfloor \frac{n-1}{j+1}\right\rfloor,\left\lfloor \frac{n}{j+1}\right\rfloor,\dots,\left\lfloor \frac{n+j-1}{j+1}\right\rfloor\right)\right)$$ holds, which proves the lower bound by the definition of $r_j(n)$. To see that the claim is true, observe that any set $U\subset V(S(\lfloor \frac{n-1}{j+1}\rfloor,\lfloor \frac{n}{j+1}\rfloor,\dots,\lfloor \frac{n+j-1}{j+1}\rfloor))$ of size $j$ is disjoint from at least one component $C$ of $S(\lfloor \frac{n-1}{j+1}\rfloor,\lfloor \frac{n}{j+1}\rfloor,\dots,\lfloor \frac{n+j-1}{\gamma+1}\rfloor) \setminus \{v\}$. 

Thus if $v \notin U$, then the leaf of $S(\lfloor \frac{n-1}{j+1}\rfloor,\lfloor \frac{n}{j+1}\rfloor,\dots,\lfloor \frac{n+j-1}{j+1}\rfloor)$ belonging to $C$ has distance at least $$1+\left\lfloor\frac{n-1}{j+1}\right\rfloor\ge \left\lfloor\frac{n}{j+1}\right\rfloor$$ from any vertex of $U$. 

If $v\in U$ holds, then $U$ is disjoint from at least two components $C_1,C_2$ of \\ $S(\lfloor \frac{n-1}{j+1}\rfloor,\lfloor \frac{n}{j+1}\rfloor,\dots,\lfloor \frac{n+j-1}{j+1}\rfloor) \setminus \{v\}$, and the leaf of $S(\lfloor \frac{n-1}{j+1}\rfloor,\lfloor \frac{n}{j+1}\rfloor,\dots,\lfloor \frac{n+j-1}{j+1}\rfloor)$ belonging to the larger path has distance at least $\lfloor\frac{n}{j+1}\rfloor$ from $v$ and thus from $U$. This completes the proof of the general lower bound.

\vspace{3mm}

To see the general upper bound, let $T$ be any tree on $n$ vertices. We will use the following claim repeatedly.

\begin{claim}
\label{clm:cut} Let $m < n$ be two positive integers. Then in any tree $T$ on $n$ vertices there exists a vertex $v$ such that if $C_1,C_2,\dots, C_s$ denote those components of $T\setminus \{v\}$ whose all vertices are at distance at most $m$ from $v$, then $\sum_{i=1}^s|C_i|\ge m$ holds.
\end{claim}

\begin{proof}[Proof of Claim \ref{clm:cut}] Let $P$ be a longest path of $T$. If $P$ contains at most $m$ vertices, then any vertex can play the role of $v$. If $P$ contains at least $m+1$ vertices, then let $v$ be the $(m+1)$st vertex from one end of $P$.
\end{proof}
For $t=1,2,\dots,j-1$ let $m_t=\lfloor \frac{n+t-1}{j+1}\rfloor$ and let $T_1=T$. We apply \clref{cut} to $T_t$ and $m_t$ for $t=1,2,\dots, j-1$ to obtain $v_t$; and then set $$T_{t+1}:=T_j\setminus \cup_{i=1}^{k_t}C_{i,t},$$ where the $C_{i,t}$ ($i=1,2,...,k_t$) are the components of $T_t\setminus \{v_t\}$ whose vertices are at distance at most $m_t$ from $v_t$. By the claim we also have $\sum_{i=1}^{k_t}|C_{i,t}|\ge m_t$. 

In this way we obtain a tree $T_j$ of at most $\lceil 2\frac{n}{j+1}\rceil$ vertices. Let $v_j$ be a vertex of $T_j$ within distance $\lceil \frac{|V(T_j)|-1}{2}\rceil$ from all vertices of $T_j$. Such a vertex exists by the result on $r_1(n)$. Clearly, $U=\{v_1,v_2,\dots,v_j\}$ is a set of vertices with ${\rm exc}_T(U)\le \lceil \frac{n}{j+1}\rceil$, which proves $r_j(T)\le \lceil \frac{n}{j+1}\rceil$.
\end{proof}

\subsection{Putting things together: the proof of \tref{dist_conn}}

Let us first prove the upper bounds of \tref{dist_conn}. To do so we introduce two types of hypergraphs with distance-$l$ domination number $\gamma$. The second construction will prove the upper bounds of \textbf{(b)}, \textbf{(c)}, and \textbf{(d)}. If $\gamma=2$, then the construction giving the smaller number of vertices depends on the values of $k$ and $l$. This is why we have the minimum of two expressions in the upper bound of \textbf{(a)}.

\vspace{2mm}

{\bf Construction 1:}

\vspace{2mm}

For $i=1,\dots, 2l(\gamma-1)+1$ let $U_i$ be pairwise disjoint sets,  and let $v_i$ and $w$ be distinct vertices which are not elements of $\bigcup_{i=1}^{2l(\gamma-1)+1}U_i$. During Construction 1 all the indices will be taken modulo $2l(\gamma-1)+1$, e.g.\ we then have $2l(\gamma-1)+2=1$.

If $k$ is odd, let $|U_i|=\frac{k-1}{2}$ for all $i$. We define a hypergraph $\cH=(V(\cH),\cE(\cH))$ in the following way. Let $$V(\cH):=\bigcup_{i=1}^{2l(\gamma-1)+1}(U_i \cup \{v_i\}),$$ 
 and let the hyperedges of $\cH$ be $$H_i:=U_i\cup U_{i+1}\cup \{v_i\}$$ for $i=1,\dots, 2l(\gamma-1)+1$. Then the size of $V(\cH)$ is $$\frac{(2l(\gamma-1)+1)(k+1)}{2}.$$
 
If $k$ is even, let $|U_{2i}|=\frac{k}{2}$ for  $i=1,\dots, l(\gamma -1)$ and $|U_{2i+1}|=\frac{k}{2}-1$ for  $i=0,\dots, l(\gamma -1)$. We define $\cH$ with the vertex set $$V(\cH):=\{w\} \cup \bigcup_{i=1}^{2l(\gamma-1)+1}(U_i \cup \{v_i\}),$$ and with the edge set $\cE(\cH):=\{H_i \mid 1 \le i \le 2l(\gamma-1)+1\}$, where $$H_i:=U_i\cup U_{i+1}\cup \{v_i\}$$ if $i=1,\dots, 2l(\gamma-1)$, and 
$$H_i:=U_i\cup U_{i+1}\cup \{v_i, w\}$$ if $i=2l(\gamma-1)+1$. Then,  $$|V(\cH)|= \frac{(2l(\gamma-1)+1)(k+1)}{2}+\frac{1}{2}=\left \lceil \frac{(2l(\gamma-1)+1)(k+1)}{2}\right\rceil.$$

To see that $\gamma_d(\cH,l)\ge \gamma$ holds in both cases, observe the following facts:
\begin{itemize}
\item 
vertex $v_i$ distance-$l$ dominates a vertex $v_j$ exactly for $$j \in \{i-l+1,...,i+l-1 \},$$
\item 
vertex $w$ distance-$l$ dominates a $v_j$ exactly for $$j \in \{2l(\gamma-1)-l+2,...,2l(\gamma-1)+l \},$$
\item
a vertex $u\in U_i$ distance-$l$ dominates a $v_j$ exactly for $$j \in \{i-l,..., i+l-1 \}.$$
\end{itemize}
So, every vertex in $V(\cH)$ distance-$l$ dominates at most $2l$ vertices $v_i$. 
This yields $\gamma_d(\cH,l) \ge \gamma$. 

\vspace{4mm}

{\bf Construction 2:}

\vspace{2mm}

This construction relies on the spider graph $S=S(a_1,a_2,\dots, a_{\gamma})$ with all of the $a_i$ being equal to $\lfloor l/2\rfloor$. Let $v$ be the only vertex of $S$ with degree $\gamma$. Let $u_1,u_2,\dots,u_\gamma$ be the neighbors of $v$ in $S$, and let $u_1',u'_2,\dots,u'_\gamma$ be the vertices of $S$ that are at distance $\lfloor l/2 \rfloor$ from $v$. 

Let $W$ be a set of size $\max\{k,\gamma\}$. Take a partition $(W_1,W_2,\dots,W_\gamma)$ of $W$ such that $|W_i|=\lfloor \frac{|W|+i-1}{\gamma}\rfloor$. Finally, for every $u \in V(S)\setminus \{v\}$, let $U_u:=U_{u,1}\hskip 0.1truecm \dot{\cup} \hskip 0.1truecm U_{u,2}$ be a set of size $k$ such that
\begin{itemize}
\item
$u \in U_{u,1}$ holds for all $u \in V(S)\setminus \{v\}$,
\item
$U_u\cap U_{u'}=\emptyset$ holds for all $u\neq u' \in V(S)\setminus \{v\}$,
\item
$U_u\cap W=\emptyset$ holds for all $u \in V(S)\setminus \{v\}$,
\item
$|U_{u,1}|=|W_i|$ for all those $u\in V(S)\setminus \{v\}$ which lie in the same component of $S\setminus \{v\}$ as $u_i$.
\end{itemize}
With the help of the previously defined sets we construct a $k$-uniform hypergraph $\cH$ in the following way, depending on the parity of $l$:

\begin{figure}
\begin{center}
\includegraphics[width=12cm]{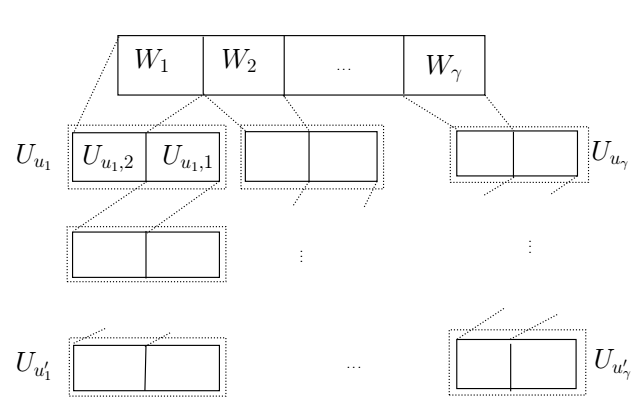}
\caption{Construction 2 in case of even $l$}
\end{center}
\end{figure}

\vskip 0.3truecm

\textsc{Case I}: $l$ is even

\vskip 0.2truecm

Let the vertex set of $\cH$ be $V(\cH)=W\cup \bigcup_{u\in V(S)\setminus \{v\} }U_u$. Thus we have $$|V(\cH)|=\frac{kl\gamma}{2}+\max\{k,\gamma\}.$$ The edge set $\cE(\cH)$ contains the following four types of hyperedges:
\begin{enumerate}
\item
all $k$-subsets of $W$, i.e. $\binom{W}{k}\subset \cE(\cH)$,
\item
for all $u \in V(S)\setminus \{v\}$, we have $U_u\in \cE(\cH)$,
\item
for all $i=1,2,\dots, \gamma$ let $W_i\cup U_{u_i,2} \in \cE(\cH)$,
\item
for every edge $(u,u')=e\in E(S)$ with $u,u'\neq v$ if $d_S(u,v)<d_S(u',v)$ holds, then let $U_{u,1}\cup U_{u',2} \in \cE(\cH)$.
\end{enumerate}

Clearly, $\cH$ is connected due to $\binom{W}{k}\subset \cE(\cH)$. We claim that $\gamma_d(\cH,l)\ge \gamma$ holds. Indeed, if $D \subset V(\cH)$ has size at most $\gamma-1$, then there exists an $i\le \gamma$ such that $$D\cap (W_i \cup \bigcup_{u  \in C_i}U_u)=\emptyset$$ holds where $C_i$ is the component of $S\setminus \{v\}$ containing $u_i$. Then $u'_i$ is at distance at least $1+2\frac{l}{2}=l+1$ from any vertex of $D$ and thus $u'_i$ is not distance $l$-dominated by $D$.

\vskip 0.4truecm

\textsc{Case II}: $l$ is odd

\vskip 0.2truecm

In addition to the sets defined above, let $Z_1,Z_2,\dots,Z_\gamma$ be pairwise disjoint sets of size $k-|W_i|$, each of which is disjoint from all previously defined sets. Let the vertex set of $\cH$ be $$V(\cH)=W\cup \bigcup_{u\in V(S)\setminus \{v\}}U_u\cup \bigcup_{i=1}^\gamma Z_i .$$ Thus we have $$|V(\cH)|\le\lceil\frac{l}{2}\rceil k\gamma.$$ As for the edge set of $\cH$, there is a fifth type of hyperedge:

\vskip 0.2truecm

5. for all $1\le i \le \gamma$ let $U_{u'_i,1}\cup Z_i \in \cE(\cH)$.

\vskip 0.2truecm

\noindent The fact that $\gamma_d(\cH,l)\ge \gamma$ follows similarly as in the previous case, because for any $(\gamma-1)$-set $D\subset V(\cH)$ there exists an $i$ such that any vertex $z \in Z_i$ is at distance at least $l+1$ from $D$.

\vspace{4mm}

Let us now turn our attention to the lower bounds. We prove first that of \textbf{(a)}. Consider a connected $k$-uniform hypergraph $\cH$ with $\gamma_d(\cH,l)\ge 2$. Let $\cM$ be a maximal matching in $\cH$ obtained in the following way. Let $$\cM_1:=\{H_1\}, \ \cI_1:=\{H\in \cE(\cH)\setminus \{H_1\}:H\cap H_1\neq \emptyset\} \textrm{ and } \cR_1:=\cE(\cH)\setminus (\cM_1\cup \cI_1).$$ Then for $s\ge 2$ we define a sequence $\cM_s,\cI_s,\cR_s$ of partitions of $\cE(\cH)$ such that: 

\vspace{2mm}

1. $\cM_s$ is a matching, 

2. every hyperedge in $\cI_s$ meets at least one hyperedge in $\cM_s$, and 

3. all hyperedges in $\cR_s$ are disjoint from all hyperedges in $\cM_s$. 

\vspace{2mm}

If $\cM_s,\cI_s,\cR_s$ are defined with $\cR_s\neq \emptyset$, then let $H_{s+1}\in \cR_s$ be a hyperedge such that $H_{s+1}\cap I_s\neq \emptyset$ for some $I_s\in\cI_s$. The existence of such $H_{s+1}$ follows from the assumption that $\cH$ is connected.  Set $$\cM_{s+1}:=\cM_s\cup \{H_{s+1}\}, \ \cI_{s+1}:=\cI_s \cup \{R\in \cR_s\setminus \{H_{s+1}\}:R\cap H_{s+1}\neq \emptyset\}$$ and $$\cR_{s+1}:=\cE(\cH)\setminus (\cM_{s+1}\cup \cI_{s+1}).$$ For the smallest positive $t$ with $\cR_t=\emptyset$, we let $\cM:=\cM_t$. Thus the size of $\cM$ is $t$.

Now let us consider the auxiliary graph $G_{\cM}$ with vertex set $\cM$ and $e=\{H_i,H_j\}\in E(G_{\cM})$ if and only if there exists $H \in \cH$ with $H\cap H_i\neq \emptyset$ and $H\cap H_j\neq \emptyset$. By the definition of $\cM$, the graph $G_{\cM}$ is connected. For a vertex $v\in \bigcup_{H\in \cM}H$ let $H_v$ denote the only element of $\cM$ containing $v$.

Suppose that for a pair $H,H'\in\cM$ we have $d_{G_{\cM}}(H,H')=r$. Then for any pair of vertices $u \in H, v\in H'$ we have $d_{\cH}(u,v) \le 1+2r$. To see this, consider the sequence $H,H_{e_1},H_{i_1},H_{e_2},H_{i_2},\dots,H_{e_r},H'$, where $e_s$ is the $s$th edge in a shortest path from $H$ to $H'$ and $H_{i_s}$ is the $s$th vertex (i.e.\ a hyperedge in $\cH$) in the same path. By the maximality of $\cM$, for every vertex $w$ of $\cH$ there exists an edge $H_w$ containing $w$ and an edge $H \in \cM$ with $H_w \cap H\neq \emptyset$, therefore by the observation above we have $$d_{\cH}(u,w) \le 2+2r_{G_{\cM}}(u)$$
for every $u \in \bigcup_{H\in\cM}H$ and $w\in V(\cH)$. 

\vspace{2mm}

If $t\ge l+1$ holds, then $|V(\cH)|\ge kt\ge  k(l+1)$, proving the desired lower bound.

\vspace{1mm}

Now suppose that $t\le l-2$ or $t=l-1$ with $t$ being odd. As we have noted, $G_{\cM}$ is connected and thus by \lref{genradius} we obtain $$r(G_{\cM})\le \left\lceil\frac{t-1}{2}\right\rceil.$$ Therefore, there exists an $H^*\in \cM=V(G_{\cM})$ such that $r_{G_{\cM}}(H^*)\le \lceil \frac{t-1}{2}\rceil$ holds and so, by the above, for a vertex $v \in H^*$ we have $$d_{\cH}(v,v') \le 2+2\left\lceil \frac{t-1}{2}\right\rceil$$ for any vertex $v'\in V(\cH)$. So in this case a vertex $v\in H^*$ distance-$l$ dominates $\cH$, contradicting $\gamma_d(\cH,l)\ge 2$.

\vspace{1mm}

If $t=l-1$ and $t$ is even, then let $T$ be a spanning tree of $G_{\cM}$. By \lref{genradius} we obtain that $T$ is a path on $t$ vertices. So we may assume that $$E(G_{\cM})\supset \{(H_i,H_{i+1}):i=1,\dots,t-1\}.$$ Let $e=(H_{t/2},H_{t/2+1})$ and consider a vertex $v\in H_{t/2}\cap H_e$. As for vertices $v'$ with $H_{v'} \cap H_i\neq \emptyset$ for some $i>t/2$, a shortest path in $\cH$ between $v$ and $v'$ need not contain $H_{t/2}$. Thus we obtain that $v$ distance-$l$ dominates $\cH$, contradicting $\gamma_d(\cH,l)\ge 2$.

\vspace{2mm}

Finally, it remains to prove the lower bound of $\textbf{(a)}$ in case of $t=l$ and thus it is enough to prove that $|V(\cH)\setminus \bigcup_{H\in\cM}H|\ge k/2$ holds. We may and will assume that the radius of $G_{\cM}$ is $\lceil\frac{l-1}{2}\rceil$. Let $T$ be a spanning tree of $G_{\cM}$. By \lref{genradius} we know that $T$ is a path if $l$ is even, and $T$ contains a path on $l-1$ vertices if $l$ is odd. We claim that even if $l$ is odd, $T$ must be a path on $t$ vertices. Indeed, otherwise any vertex $v \in H_e$ distance-$l$ dominates $\cH$ where $e$ is the middle edge of a path on $l-1$ vertices that is contained in $T$. This would contradict $\gamma_d(\cH,l)\ge 2$. By this we may assume that $E(G_{\cM})\supset \{(H_i,H_{i+1}):i=1,\dots,l-1\}$.

\vspace{2mm}

\begin{claim}
\label{claim:disjoint}
We have the following:

\vspace{2mm}

\textbf{(i)} For any pair of edges $e,e'$ in $T$ we have $H_e\cap H_{e'}=\emptyset$.

\vspace{1mm}

\textbf{(ii)} There exist $w,w' \in V(\cH)\setminus \bigcup_{H \in \cM}H$ and $H_{w}, H_{w'} \in \cE(\cH)$ with $$w \in H_w \textrm{ and} \ w'\in H_{w'},$$ such that $H_{w}$ meets only $H_1$ and $H_{w'}$ meets only $H_l$, moreover $H_{w}$ and $H_{w'}$ are disjoint from all the other $H \in \cM$ and also from $H_e$ for all $e\in E(T)$.
\end{claim}

\begin{proof}[Proof of Claim] We have two cases depending on the parity of $l$.

\vspace{2mm}

\textsc{Case I}: $l$ is even.

\vskip 0.2truecm

Now we prove \textbf{(i)} in this case. Suppose that $H_{e_i}\cap H_{e_j}\neq \emptyset$ with $e_i=(H_i,H_{i+1}), e_j=(H_j,H_{j+1})$. If $i<j\le l/2$, then a vertex $v \in H_{e_{l/2}}\cap H_{l/2+1}$ distance-$l$ dominates $\cH$, contradicting $\gamma_{dist}(\cH,l)\ge 2$.. Similarly, if $i<j$ and $j\ge l/2$, then a vertex $v \in H_{e_{l/2}}\cap H_{l/2}$ distance-$l$ dominates $\cH$, contradicting $\gamma_{dist}(\cH,l)\ge 2$. Also, if $i<l/2<j$, then if $l/2-i\le j-l/2$, then a vertex $v$ from $H_{l/2-1}\cap H_{e_{l/2-1}}$ distance-$l$ dominates $\cH$, while if $l/2-i\ge j-l/2$, then a vertex $v$ from $H_{l/2+2}\cap H_{e_{l/2+1}}$ distance-$l$ dominates $\cH$, contradicting $\gamma_{dist}(\cH,l)\ge 2$. We are done with \textbf{(i)} in \textsc{Case I}.

To see \textbf{(ii)} suppose that, for every $w\in V(\cH)\setminus \bigcup_{H\in\cM}H$ and $H_w$ containing $w$, the hyperedge $H_w$ meets $H_e$ for some $e\in E(T)$ or $H_w$ meets some $H_z$ with $z\ge 2$. Then a vertex in $H_{e_{n/2}}\cap H_{n/2+1}$ distance-$l$ dominates $\cH$, contradicting $\gamma_{dist}(\cH,l)\ge 2$. The existence of $w'$ and $H_{w'}$ can be shown analogously. This proves \textbf{(ii)} in \textsc{Case I}.

\vskip 0.2truecm

\textsc{Case II}: $l$ is odd.

\vskip 0.2truecm

The proof of this case is very similar to the previous one. Let us just show \textbf{(ii)}. Suppose that, for every $w\in V(\cH)\setminus \bigcup_{H\in\cM}H$ and $H_w$ containing $w$, the hyperedge $H_w$ meets $H_e$ for some $e\in E(T)$ or $H_w$ meets some $H_z$ with $z\ge 2$. Then a vertex in $H_{e_{\lceil n/2 \rceil}}\cap H_{\lceil n/2\rceil}$ distance-$l$ dominates $\cH$, contradicting $\gamma_{dist}(\cH,l)\ge 2$.
\end{proof}

Note that $H_w\cap H_{w'} \subset V(\cH)\setminus \bigcup_{H \in \cM}H$ and also $H_w\cup H_{w'}\cup \bigcup_{e\in E(T)}H_e \subset V(\cH)$, and thus writing $I=|H_w\cap H_{w'}|$ we obtain $|V(\cH)|\ge \max\{lk+I,(l+1)k-I\} \ge lk+k/2$. This finishes the proof of the lower bound of $\textbf{(a)}$.

\vskip 0.4truecm

Next we prove the lower bound of \textbf{(b)}. We will need the following lemma.

\begin{lemma}
\label{lem:gen_low}
For any $\gamma, l \ge 2$, let $t^*$ denote the smallest $t$ with $r_{\gamma-1}(t)\ge \frac{l-1}{2}$. Then we have
$$n_{dc}(k,\gamma,l)\ge t^*k.$$
\end{lemma}

\begin{proof}
Let $\cH$ be a connected $k$-uniform hypergraph with $\gamma_d(\cH,l)\ge \gamma$. Let $\cM$ be a maximal matching in $\cH$ obtained as in the proof of the lower bound of part \textbf{(a)}, and let us consider the auxiliary graph $G_{\cM}$. For a vertex $v\in \bigcup_{H\in \cM}H$ let $H_v$ denote the only element of $\cM$ containing $v$. Let the size of $\cM$ be $t$. We assume first that $t<t^*$, what means $r_{\gamma-1}(t)<\frac{l-1}{2}$.

Suppose that for a pair $H,H'\in\cM$ we have $d_{G_{\cM}}(H,H')=r$. Then for any pair of vertices $u \in H, v\in H'$ we have $d_{\cH}(u,v) \le 1+2r$. To see this, consider the sequence $H,H_{e_1},H_{i_1},H_{e_2},H_{i_2},\dots,H_{e_r},H'$, where $e_s$ is the $s$th edge in a shortest path from $H$ to $H'$ and $H_{i_s}$ is the $s$th vertex (i.e.\ a hyperedge in $\cH$) in the same path. 
Let $\cU \subset \cM$ be a subset of size $\gamma-1$ with $r_{G_{\cM}}(\cU)=r_{\gamma-1}(G_{\cM})\le r_{\gamma-1}(t)$, and let $L \subset V(\cH)$ be a set containing one vertex from each $U \in \cU$.

By the maximality of $\cM$, for every vertex $w$ of $\cH$ there exist an edge $H_w$ containing $w$ and an edge $H \in \cM$ with $H_w \cap H\neq \emptyset$. Therefore by the observation above and by the definition of $\cU$, there exist a $U \in \cU$ and a vertex $u \in U$ for which we have $$d_{\cH}(u,w) \le 2+2r_{G_{\cM}}(U)\le 2+2r_{\gamma-1}(t)<2+2\frac{l-1}{2}=l+1.$$
This means that if $t<t^*$ holds, then the $(\gamma-1)$-subset $L$ distance-$l$ dominates $\cH$. Therefore $\cM$ consists of at least $t^*$ hyperedges and thus $|V(\cH)| \ge t^*k$ holds.
\end{proof}

\noindent The lower bound of \textbf{(b)} follows by applying \lref{genradius} with $j=\gamma-1$ together with \lref{gen_low}, noting that $\lceil\frac{t}{\gamma} \rceil\ge \frac{l-1}{2}$ implies $\frac{t}{\gamma}>\frac{l-1}{2}-1$.

\vskip 0.3truecm

Finally, we prove the lower bound of \textbf{(c)} and \textbf{(d)}. This will follow from the claim that any maximal matching in the edge set $\cE(\cH)$ of a connected hypergraph $\cH$ with $\gamma_d(\cH,2)\ge \gamma$ has size at least $\gamma$. To see this suppose that $\cM=\{H_1,H_2,\dots, H_m\}$ is a maximal matching in $\cE(\cH)$ and for any $i=1,2,\dots, m$ let $v_i$ be a vertex of $H_i$. As any vertex $v\in V(\cH)$ is contained in a hyperedge $H_v$ which, by maximality of  $\cM$, intersects some $H_i  \in \cM$, the set $D=\{v_i:i=1,2,\dots,m\}$ distance-$2$ dominates $\cH$. Therefore $m\ge \gamma$ must hold as claimed.
\hfill $\qed$

\section{Final remarks and open problems}
\label{sec:discuss}
We addressed the problem of finding the minimum number of vertices that a connected $k$-uniform hypergraph with high domination number must contain, and we considered two main variants of the problem. For the original notion of domination and for $s$-wise domination we found general lower and upper bounds on $n(k,\gamma,s)$ in which even the order of magnitude of the second term matches. The natural open problem occurs: it can be of interest to find the constant coefficient of this second term.

\tref{dist_conn}, our main result concerning distance domination determines the asymptotics of $n_{dc}(k,\gamma,l)$ if $k$ and $\gamma$ are fixed and $l$ tends to infinity, or if all three parameters tend to infinity. Closing the gap of roughly $2k\gamma$ between the upper and lower bounds remains an interesting open problem.

We had a good reason to choose the notion of Berge paths in the definition of distance-$l$ domination. The most common other definitions of a path in hypergraphs are \textit{linear paths}, where two consecutive hyperedges of the path must share exactly one vertex (an even more restrictive notion is a \textit{loose path}) and \textit{tight paths} where the vertices $v_1,v_2,\dots, v_{k+l-1}$ of the path should be chosen in such a way that the $i$th hyperedge of the path is $\{v_i,v_{i+1},\dots, v_{i+k-1}\}$ for all $i=1,2,\dots,l$. This implies that consecutive hyperedges of a tight path share $k-1$ vertices. Note that in the construction showing the upper bound of \tref{main} no pair of hyperedges has intersection size 1 or $k-1$, therefore the construction does not contain linear or tight paths of length larger than 1 and thus distance domination would not differ from ordinary domination, had we used these notions of hypergraph paths to define distance.

There are various results on different domination numbers of a hypergraph in the literature: on the $s$-domination number in \cite{A2}, on the inverse domination number in \cite {JT}, on the total domination number in \cite{BHTY}, and on the connection of the domination number with the transversal number in \cite{ABBT}, \cite{BHT}. Let us finish with the following theorem that can be obtained simply by rearranging the lower bound of \tref{dist_conn}. In the style of Meir and Moon \cite{MM}, it uses only the size of the vertex set, the prescribed distance bound $l$, and the uniformity of $\cH$.

\begin{theorem}
\label{thm:gamma} If $\cH$ is a connected $k$-uniform hypegraph with $|V(\cH)|=n$, then
$$\gamma_{dc}(\cH,l) \le \begin{cases}
    \frac{n}{k}       & \quad \text{if } l= 2,3\text{ or}\ 4,\\
    \frac{n}{k}\cdot\frac{2}{l-3}  & \quad \text{if } l>4 .\\
  \end{cases}$$
\end{theorem}

It remains an open problem to make these upper bounds tight.


\begin{thebibliography}{99}

\bibitem{A1}
\textsc{B. D. Acharya}, \textit{Domination in hypergraphs}, AKCE International Journal of Graphs and Combinatorics, 4 (2007), pp. 117--126.

\bibitem{A2}
\textsc{B. D. Acharya}, \textit{Domination in hypergraphs II}, New directions, Proceedings of ICDM, Mysore, India, 2008, Ramanujan Mathematical Society Lecture Notes Series, 13 (2010), pp. 1--18.

\bibitem{ABBT}
\textsc{S. Arumugam, B. K. Jose, Cs. Bujt\'as, and Zs. Tuza}, \textit{Equality of domination and transversal numbers in hypergraphs}, Discrete Applied Mathematics, 161 (2013), pp. 1859--1867.

\bibitem{BHT}
\textsc{Cs. Bujt\'as, M. A. Henning, and Zs. Tuza}, \textit{Transversals and domination in uniform hypergraphs}, European Journal of Combinatorics, 33 (2012), pp. 62--71.

\bibitem{BHTY}
\textsc{Cs. Bujt\'as, M. A. Henning,  Zs. Tuza, and A. Yeo}, \textit{Total transversals and total domination in uniform hypergraphs}, The Electronic Journal of Combinatorics, 21(2014), \#P2.24.

\bibitem{DFHK}
\textsc{R. Davila, C. Fast, M. A. Henning, and F. Kenter}, \textit{Lower bounds on the distance domination number of a graph}, arXiv:1507.08745

\bibitem{FJ}
\textsc{J. F. Fink and M. S. Jacobson}, \textit{On $n$-domination, $n$-dependence and forbidden subgraphs},
 In: Graph Theory with Applications to Algorithms and Computer Science, Wiley, New
York (1985),  pp. 301--311.

\bibitem{Getal}
\textsc{D. Gerbner, B. Keszegh, D. P\'alv\"olgyi, B. Patk\'os, M. Vizer, and G. Wiener}, \textit{Finding a majority ball with majority answers}, arXiv:1509.08276

\bibitem{HH}
\textsc{F. Harary and T. W. Haynes}, \textit{Nordhaus-Gaddum inequalities for domination in graphs, Discrete Mathematics}, 155 (1996), pp. 99-–105.

\bibitem{HHS1}
\textsc{T. W. Haynes, S. T. Hedetniemi, and P. J. Slater (eds)}, \textit{Fundamentals of Domination
in Graphs}, Marcel Dekker, Inc. New York, 1998.

\bibitem{HHS2}
\textsc{T. W. Haynes, S. T. Hedetniemi, and P. J. Slater (eds)}, \textit{Domination in Graphs:
Advanced Topics}, Marcel Dekker, Inc. New York, 1998.

\bibitem{H}
\textsc{M. A. Henning}, \textit{Distance domination in graphs}, Domination in Graphs: Advanced Topics, T.W. Haynes, S.T. Hedetniemi, and P.J. Slater (eds), Marcel Dekker, Inc. New York, (1998), pp. 335--365.

\bibitem{HL}
\textsc{ M. A. Henning, N. Lichiardopol}, \textit{Distance domination in graphs with given minimum and maximum degree}, manuscript.

\bibitem{HLo}	
\textsc{M. A. Henning, C. L\"owenstein}, \textit{Hypergraphs with large domination number and with edge sizes at least three}, Discrete Applied Mathematics,  160 (2012), pp. 1757--1765.	


\bibitem{Hirsch} \textsc{J.~W.~P.~Hirschfeld}, \textit{Projective geometries over finite fields},  Clarendon Press, Oxford, 1979, 2nd edition, 1998.

\bibitem{JT}
\textsc{B. K. Jose, Zs. Tuza}, \textit{Hypergraph domination and strong independence}, Applicable Analysis and Discrete Mathematics, 3 (2009), pp. 347--358.

\bibitem{MM}
\textsc{A. Meir and J. W. Moon}, \textit{Relations between packing and covering number of a tree}, Pacific Journal of Mathematics, 61 (1975), pp. 225--233.

\bibitem{NP}
\textsc{Z.L. Nagy, B. Patk\'os}, \textit{On the number of maximal intersecting k-uniform families and further applications of Tuza's set pair method}, The Electronic Journal of Combinatorics, 22 (2015), \#P1.83.

\bibitem{T1}
\textsc{Zs. Tuza}, \textit{Critical hypergraphs and intersecting set-pair systems}, Journal of Combinatorial Theory, Series B, 39 (1985), pp. 134--145.

\bibitem{T2}
\textsc{Zs. Tuza}, \textit{Inequalities for two set systems with prescribed intersections}, Graphs and
Combinatorics, 3 (1987), pp. 75--80.
\end{thebibliography}
\end{document}